\documentclass[11pt]{amsart}
\usepackage{mathrsfs}
\usepackage{amscd}
\usepackage[mathscr]{eucal}
\usepackage{amsfonts}
\usepackage{amsmath}
\usepackage{amsthm}
\usepackage{amssymb}
\usepackage{latexsym}
\usepackage[all]{xy}
\usepackage{array}
\usepackage{hyperref}

\theoremstyle{plain}

\newtheorem{lemma}[equation]{Lemma}
\newtheorem{proposition}[equation]{Proposition}
\newtheorem{theorem}[equation]{Theorem}
\theoremstyle{definition}
\newtheorem{definition}[equation]{Definition}
\newtheorem{example}[equation]{Example}
\newtheorem{remark}[equation]{Remark}
\newtheorem{notation}[equation]{Notation}

\numberwithin{equation}{section}

\newcommand{\Ai}{{\mathrm{A}_\infty}}
\newcommand{\End}{{\mathrm{End}}}
\newcommand{\F}{{\mathbb{F}}}
\newcommand{\Hom}{{\mathrm{Hom}}}
\newcommand{\Ob}{{\mathrm{Ob}}}
\newcommand{\Stab}{{\mathrm{Stab}}}
\newcommand{\kk}{{\mathbf{k}}}
\newcommand{\Mod}{{\mathrm{Mod}}}
\newcommand{\OO}{{\mathscr{O}}}
\newcommand{\Ind}{{\mathrm{Ind}}}
\newcommand{\Res}{{\mathrm{Res}}}
\newcommand{\Z}{{\mathbb{Z}}}

\title{Noetherian property of infinite EI categories}

\author{Wee Liang Gan} 
\address{Department of Mathematics, University of California, Riverside, CA 92521, USA.}
\email{wlgan@math.ucr.edu}

\author{Liping Li}
\address{Department of Mathematics, University of California, Riverside, CA 92521, USA.}
\email{lipingli@math.ucr.edu}

\begin{document}

\begin{abstract}
It is known that finitely generated FI-modules 
over a field of characteristic 0 are Noetherian. 
We generalize this result to the abstract setting 
of an infinite EI category satisfying certain combinatorial conditions.
\end{abstract}

\maketitle

\section{Introduction}

Let FI be the category whose objects are finite sets and morphisms are injections. 
If an object of FI is a set with $i$ elements, then its automorphism group
is isomorphic to the symmetric group $S_i$; thus, a module of the category 
FI gives rise to a sequence of representations  $V_i$ of $S_i$. 
Our present paper is motivated by the theory of FI-modules developed by 
Church, Ellenberg, and Farb in \cite{CEF} in order to study stability phenomenon 
in sequences of representations of the symmetric groups; see also \cite{F}.
The following theorem plays a fundamental role in the theory of FI-modules in \cite{CEF}.

\begin{theorem} \label{fi}
Any finitely generated $\mathrm{FI}$-module over a field $\kk$ of 
characteristic 0 is Noetherian. 
\end{theorem}

Theorem \ref{fi} was proved by Church, Ellenberg, and Farb 
in \cite[Theorem 1.3]{CEF}. Subsequently, a generalization of Theorem \ref{fi}
to any Noetherian ring $\kk$ was given by Church, Ellenberg, Farb, and Nagpal  
in \cite[Theorem A]{CEFN}.
Using their result, Wilson \cite[Theorem 4.21]{Wi} deduced 
analogous theorems for the categories FI$_{BC}$ and FI$_D$ associated 
to the Weyl groups of type $B/C$ and $D$, respectively.
The categories FI, FI$_{BC}$, and FI$_D$ 
are examples of \emph{infinite} EI categories,
in the sense of the following definition.

\begin{definition}
An \emph{EI category} is a small category in which every endomorphism is an 
isomorphism. An EI category is said to be \emph{finite} (resp. \emph{infinite}) 
if its set of morphisms is finite (resp. infinite).
\end{definition}

L\"uck \cite[Lemma 16.10]{L} proved a version of Theorem \ref{fi} for finite EI categories 
when $\kk$ is any Noetherian ring. However, there is no such general result for
EI categories which are infinite. The goal of our present paper is to find
general sufficiency conditions on an infinite EI category so that the analog of Theorem \ref{fi} holds.
After recalling some basic facts on EI categories in Section  \ref{generalities},
we state our main result in Section  \ref{statement} and give the proof in Section \ref{proofs}.
We make some further remarks in Section \ref{further-remarks}.

Our main result (Theorem \ref{maintheorem}) gives a generalization 
of Theorem \ref{fi} to
the abstract setting of an infinite EI category satisfying certain 
simple combinatorial conditions.
It is applicable to the categories FI, FI$_{BC}$, and FI$_D$.
It is also applicable to the category VI
of finite dimensional $\F_q$-vector spaces and linear injections, where $\F_q$
denotes the finite field with $q$ elements.

In contrast to \cite{CEF}, the symmetric groups do not play 
any special role in our paper, and we do not use any results  from the
representation theory of symmetric groups in our proofs.
We hope to extend the theory of FI-modules and representation stability 
to our framework.
It should also be interesting to find conditions under which our main result
extends to an arbitrary Noetherian ring $\kk$.
The proof of \cite[Theorem A]{CEFN} makes use of \cite[Proposition 2.12]{CEFN}
which does not hold when the category FI is replaced by the category VI.

\begin{remark}
A generalization of Theorem \ref{fi} was proved 
by Snowden \cite[Theorem 2.3]{Sn}
in the language of twisted commutative algebras;
see \cite[Proposition 1.3.5]{SS1}.
His result has little overlap with our Theorem \ref{maintheorem}. 
Indeed, to interpret the category of modules of an EI category $C$ of
type $\Ai$ (in the sense of Definition \ref{A-infinity})
as the category of modules of a twisted commutative algebra
finitely generated in order 1,
the automorphism group of any object $i$ of $C$ must 
necessarily be the symmetric group $S_i$.
\end{remark}

\begin{remark}
After this paper was written, we were informed by  
Putman and Sam that they have a very recent preprint \cite{PS}
which proved the analog of Theorem \ref{fi} when $\kk$ is any Noetherian ring 
for several examples of {\em linear-algebraic} type categories 
such as the category VI. 
We were also informed by Sam and Snowden that they have
a very recent preprint \cite{SS2} which gives combinatorial criteria for 
representations of categories (not necessarily EI) to be Noetherian.
Their combinatorial criteria are very different from our conditions; 
in particular, it does not seem that Theorem \ref{maintheorem} will follow
from their results.
\end{remark}

\section{Generalities on EI categories}  \label{generalities}

Let $C$ be an EI category.
We shall assume throughout this paper that $C$ is skeletal.

\subsection{Quiver underlying an EI category}
We denote by $\Ob(C)$ the set of objects of $C$.
For any $i,j \in \Ob(C)$, we write $C(i,j)$ for the set of morphisms 
from $i$ to $j$. 

\begin{definition} 
A morphism $\alpha$ of $C$ is called 
\emph{unfactorizable} if:
\begin{itemize}
\item
 $\alpha$ is not an isomorphism;

\item
whenever $\alpha=\beta_1\beta_2$ where $\beta_1$ and $\beta_2$ 
are morphisms of $C$, either $\beta_1$ or $\beta_2$ is an isomorphism.
\end{itemize}
\end{definition}

We define a quiver $Q$ associated to $C$ as follows. The set of vertices of $Q$ is $\Ob(C)$.
The number of arrows from a vertex $i$ to a vertex $j$ is 
1 if there exists an unfactorizable morphism from $i$ to $j$; 
it is 0 otherwise. We call $Q$ the quiver underlying
the EI category $C$. 

\begin{definition}  \label{A-infinity}
Let $\Z_+$ be the set of non-negative integers.
We say that $C$ is an EI category of type $\Ai$ if
$\Ob(C) = \Z_+$ and the quiver underlying $C$ is:
\begin{equation*}
0 \longrightarrow 1 \longrightarrow 2 \longrightarrow \cdots.
\end{equation*}
\end{definition}

\subsection{Finiteness conditions}

Recall that
the EI category $C$ is finite (resp. infinite) if the set of morphisms of $C$ 
is finite (resp. infinite).

\begin{definition}
We say that $C$ is \emph{locally finite} if $C(i,j)$ is a finite set for all
$i,j\in\Ob(C)$. 
\end{definition}

There is a partial order $\le $ on $\Ob(C)$ defined by $i\le j $ if $C(i,j)$ is nonempty.

\begin{definition}
We say that $C$ is \emph{strongly locally finite} if it is 
locally finite and for every $i,j \in \Ob(C)$ with $i\le j$, there are 
only finitely many $l \in\Ob(C)$ such that $i\le l \le j$.
\end{definition}

Observe that $C$ is strongly locally finite if and only if for every
$i,j\in \Ob(C)$ with $i\le j$, the full subcategory of $C$ generated by
all objects $l$ satisfying $i\le l\le j$ is a finite EI category.
Any locally finite EI category of type $\Ai$ is strongly locally finite.

\begin{remark} \label{factorization}
If $C$ is strongly locally finite, then any morphism $\alpha$
of $C$ which is not an isomorphism can be written as 
$\alpha = \beta_1 \cdots \beta_r$  where $\beta_1, \ldots, \beta_r$ 
are unfactorizable morphisms; 
moreover, in this case, 
for any $i,j\in \Ob(C)$, one has $i\le j$ if and only if there exists a 
directed path in $Q$ from $i$ to $j$.
\end{remark}

\subsection{Modules}
Let $\kk$ be a commutative ring. 

For any set $X$, we shall write
$\kk X$ for the free $\kk$-module with basis $X$.
If $G$ is a group, then
$\kk G$ is the group algebra of $G$ over $\kk$.

\begin{definition}
The \emph{category algebra} of $C$ over $\kk$ is 
the $\kk$-algebra $C(\kk)$ defined by
\begin{equation*}
C(\kk)  = \bigoplus_{i,j\in\Ob(C)} \kk C(i,j). 
\end{equation*}
If $\alpha\in C(i,j)$ and $\beta\in C(l,m)$, their product in $C(\kk)$ 
is defined to be the composition $\alpha\beta$ if $i=m$; it is defined
to be 0 if $i\neq m$.
\end{definition}

For any $i\in \Ob(C)$, we denote by $G_i$ the group $C(i,i)$,
and write $\mathbf{1}_i$ for the identity element of the group $G_i$.

\begin{definition}
A $C(\kk)$-module $V$ is called a \emph{$\kk C$-module} if it is
\emph{graded}, i.e. if $V$ is equal 
as a $\kk$-module to the direct sum
\begin{equation*}
V = \bigoplus_{i\in \Ob(C)} V_i, 
\end{equation*}
where $V_i = \mathbf{1}_i V$ for all $i\in \Ob(C)$.
\end{definition}

Equivalently, one can define a $\kk C$-module to be a 
covariant functor from $C$
to the category of $\kk$-modules. 
We shall write $\Mod_\kk(C)$ for the 
category of $\kk C$-modules; in particular, for a group $G$, we
write $\Mod_\kk(G)$ for the category of $\kk G$-modules.
The category $\Mod_\kk(C)$ is an abelian category.

\begin{definition}
Let $V$ be a $\kk C$-module. 
An element $v\in V$ is \emph{homogeneous} if there exists $i\in \Ob(C)$
such that $v\in V_i$; we call $i$ the \emph{degree} of $v$, and denote
it by $\deg v$.
\end{definition}

For any $i\in \Ob(C)$, we have the restriction functor
\begin{equation*}
 \Res : \Mod_\kk(C) \longrightarrow \Mod_\kk(G_i), \quad V \mapsto V_i.
\end{equation*}
The restriction functor is exact and it has a left adjoint,
the induction functor
\begin{equation*}
 \Ind : \Mod_\kk(G_i) \longrightarrow \Mod_\kk(C), \quad U \mapsto
 \bigoplus_{j\in\Ob(C)} \kk C(i,j)\otimes_{\kk G_i}  U.
\end{equation*}

\begin{notation}  
For any $i\in \Ob(C)$, let $M(i) = \Ind (\kk G_i)$.
\end{notation}
Thus,
\begin{equation*}
M(i) = \bigoplus_{j\in \Ob(C)} M(i)_j, \quad\mbox{ where }\quad 
M(i)_j = \kk C(i,j). 
\end{equation*}
In particular, note that $M(i) = C(\kk) \mathbf{1}_i$.

\begin{remark}
It is easy to see that $M(i)$ is a projective $\kk C$-module. Indeed, since $\Ind$ has an exact right adjoint functor, it takes projectives to projectives.
\end{remark}

\subsection{Finitely generated modules}

If $V$ is a $\kk C$-module and $s$ is an element of
$V$ homogeneous of degree $i$, we have a homomorphism 
\begin{equation*}
\pi_s : M(i) \longrightarrow V 
\end{equation*}
defined by $\pi_s (\alpha) = \alpha s$ for all $\alpha\in C(i,j)$,
for all $j\in \Ob(C)$. 
If $S$ is a subset of $V$ and all elements of $S$ are homogeneous,
then we have a homomorphism 
\begin{equation} \label{pi_S} 
\pi_S : \bigoplus_{s\in S} M(\deg s) \longrightarrow V 
\end{equation}
whose restriction to the component corresponding to $s$ is $\pi_s$. 
The image of $\pi_S$ is the $\kk C$-submodule of $V$ generated by
$S$.

\begin{notation}  \label{M(S)}
If $V$ is a $\kk C$-module and $S$ is a set of homogeneous elements of $V$, let 
\begin{equation*}
 M(S) = \bigoplus_{s\in S} M(\deg s). 
\end{equation*}
\end{notation}

\begin{definition}
A set $S$ is called a set of \emph{generators} of 
a $\kk C$-module $V$ if $S\subset V$ and the only $\kk C$-submodule 
of $V$ containing $S$ is $V$ itself.
A $\kk C$-module $V$ is \emph{finitely generated} if it has a finite set of
generators.
\end{definition}

A set $S$ of generators of a $\kk C$-module $V$ is said to be a
set of homogeneous generators if all the elements of $S$ are homogeneous. 
Clearly, $V$ is finitely generated if and only if it has a finite set
of homogeneous generators. Hence, we have:

\begin{lemma}  \label{fg}
A $\kk C$-module $V$ is finitely generated if and only if there exists
a finite set $S$ of homogeneous elements of $V$ such that the
homomorphism $\pi_S: M(S) \to V$ of (\ref{pi_S}) is surjective.  
\end{lemma}

Note that if $C$ is locally finite and $V$ is a finitely generated
$\kk C$-module, then $V_i$ is finite dimensional for all $i\in \Ob(C)$.

\begin{definition}
A $\kk C$-module $V$ is \emph{Noetherian} if every $\kk C$-submodule of
$V$ is finitely generated.
\end{definition}

Equivalently, a $\kk C$-module is Noetherian if it satisfies the ascending chain condition on its $\kk C$-submodules.

\section{Noetherian property of finitely generated modules}  \label{statement}

We assume in this section that $C$ is a locally finite EI category of type $\Ai$.

\subsection{Transitivity condition}
We say that $C$ satisfies the \emph{transitivity condition} if
for each $i\in \Z_+$, the action of $G_{i+1}$ on $C(i, i+1)$ is transitive.

\begin{lemma} \label{transitivity}
If $C$ satisfies the transitivity condition, 
then for any $i,j\in \Z_+$ with $i<j$, the $G_j$ action on $C(i,j)$
is transitive.
\end{lemma}
\begin{proof}
Let $\alpha, \alpha' \in C(i,j)$. 
Since $C$ is of type $\Ai$, there exists two sequences of morphisms 
\begin{equation*}
\alpha_r, \alpha'_r\in C(i+r-1,i+r) \quad\mbox{ for }\quad r=1,\ldots,j-i, 
\end{equation*}
such that $\alpha = \alpha_{j-i}\cdots \alpha_1$ and
$\alpha' = \alpha'_{j-i}\cdots \alpha'_1$.
There exists $g_1\in G_{i+1}$ such that $g_1\alpha_1 = \alpha'_1$.
We find, inductively, an element $g_r\in G_{i+r}$ such that 
$g_r\alpha_r = \alpha'_r g_{r-1}$ for $r= 2, \ldots, j-i$.
Then one has $g_{j-i} \alpha = \alpha'$.
\end{proof}

\subsection{Bijectivity condition}
Suppose that $C$ satisfies the transitivity condition.

\begin{notation}
For each $i\in \Z_+$, we choose and fix a morphism 
\begin{equation*}
\alpha_i\in C(i,i+1) .
\end{equation*}
Moreover, for any $i, j\in \Z_+$ with $i<j$, let
\begin{equation*}
\alpha_{i,j}  = \alpha_{j-1} \cdots \alpha_{i+1} \alpha_i \in C(i,j) . 
\end{equation*}
Let $H_{i,j} = \Stab_{G_j} (\alpha_{i,j})$, and define the map
 \begin{equation} \label{mapm}
 m_{i,j} :  C(i,j) \longrightarrow C(i,j+1) ,\quad  \gamma \mapsto \alpha_j \gamma .
 \end{equation}
\end{notation}

Suppose $h\in H_{i,j}$. 
By the transitivity condition, there exists $g \in G_{j+1}$ such that
$g\alpha_j = \alpha_j h$.  We have
\begin{equation*}
 g\alpha_{i,j+1} = g\alpha_j \alpha_{i,j} = \alpha_j h \alpha_{i,j} = \alpha_j \alpha_{i,j}
=\alpha_{i,j+1},  
\end{equation*}
and hence $g\in H_{i,j+1}$.
Now, for any $\gamma\in C(i,j)$,
\begin{equation*} 
m_{i,j}(h\gamma) = \alpha_j h \gamma 
=  g \alpha_j \gamma
 = g m_{i,j}(\gamma).
\end{equation*}
It follows that $m_{i,j}$ maps each $H_{i,j}$-orbit $\OO$
in $C(i,j)$  into 
a $H_{i,j+1}$-orbit in $C(i,j+1)$;
in particular, we get a map on the set of orbits:
\begin{equation} \label{mapmu}
\mu_{i,j}:  H_{i,j}  \backslash C(i,j)  \longrightarrow  H_{i,j+1} \backslash C(i,j+1), 
\end{equation}
where $\mu_{i,j}(\OO)$ is the $H_{i,j+1}$-orbit that contains $m_{i,j}(\OO)$.

We say that $C$ satisfies the \emph{bijectivity condition} if, for each $i\in \Z_+$,
the map $\mu_{i,j}$ is bijective for all $j$ sufficiently large.
It is clear that this is independent of the choice of the maps $\alpha_i$.

\begin{remark} \label{double coset formulation of bijectivity condition}
For $i<j$, we have a bijection
\begin{equation*}
G_j/ H_{i,j} \longrightarrow C(i,j), \quad g H_{i,j} \mapsto g\alpha_{i,j}.
\end{equation*}
Identifying $C(i,j)$ with $G_j/ H_{i,j}$ via this bijection, the maps (\ref{mapm}) and (\ref{mapmu}) are, respectively, 
\begin{equation*}
m'_{i,j} : G_j/ H_{i,j} \to G_{j+1}/ H_{i, j+1}, \quad gH_{i,j} \mapsto u H_{i,j+1},
\end{equation*}
and
\begin{equation*} 
\mu'_{i,j} : H_{i,j} \backslash G_{i,j} / H_{i,j} \to H_{i, j+1} \backslash G_{i,j+1} / H_{i,j+1}, \quad H_{i,j} g H_{i,j} \mapsto H_{i,j+1} u H_{i,j+1},
\end{equation*}
where $u\in G_{j+1}$ is any element such that $\alpha_j g = u \alpha_j$. The bijectivity condition is equivalent to the condition that, for each $i\in \Z_+$, the map $\mu'_{i,j}$ is bijective for all $j$ sufficiently large. (The set $H\backslash G/ H$ where $H$ is a subgroup of a finite group $G$ appears naturally in the theory of Hecke algebras, see \cite{K}.)
\end{remark}

\begin{lemma} \label{bijective-implies-injective}
Assume that $C$ satisfies the transitivity and bijectivity conditions.
Then for each $i\in \Z_+$,
the map $m_{i,j}$ is injective for all $j$ sufficiently large.
\end{lemma}

\begin{proof}
Let $j$ be an integer such that $\mu_{i,j}$ is injective.
Suppose $\gamma_1, \gamma_2 \in C(i,j)$ and $m_{i,j}(\gamma_1)=m_{i,j}(\gamma_2)$.
By Lemma \ref{transitivity},
there exists $g_1, g_2\in G_j$ such that $\gamma_1= g_1 \alpha_{i,j}$ and
$\gamma_2 = g_2 \alpha_{i,j}$. 
There also exists $g\in G_{j+1}$ such that $g\alpha_j = \alpha_j g_1$. One has
\begin{equation*}
g \alpha_j \alpha_{i,j} = \alpha_j g_1 \alpha_{i,j} = \alpha_j \gamma_1
= \alpha_j \gamma_2 = \alpha_j g_2 \alpha_{i,j} 
= g \alpha_j g_1^{-1} g_2  \alpha_{i,j} , 
\end{equation*}
and hence 
$m_{i,j}(\alpha_{i,j}) = m_{i,j} (g_1^{-1} g_2 \alpha_{i,j})$.
It follows, by the injectivity of $\mu_{i,j}$, that $\alpha_{i,j}$ and 
$g_1^{-1} g_2 \alpha_{i,j}$ are in the same $H_{i,j}$-orbit.
Thus, there exists $h\in H_{i,j}$ such that
$  h \alpha_{i,j} = g_1^{-1} g_2 \alpha_{i,j}$.
But $h \alpha_{i,j} = \alpha_{i,j}$, so
$\alpha_{i,j} = g_1^{-1} g_2 \alpha_{i,j}$, and hence
$\gamma_1 = g_1\alpha_{i,j} = g_2 \alpha_{i,j} = \gamma_2$.
\end{proof}

\subsection{Main result} 
We shall give the proof of the following theorem in the next section.

\begin{theorem}  \label{maintheorem}
Assume that
$C$ satisfies the transitivity and bijectivity
conditions, and $\kk$ is a field of characteristic 0.
Let $V$ be a finitely generated $\kk C$-module. 
Then $V$ is a Noetherian $\kk C$-module.
\end{theorem}

Let us give some examples of categories $C$ with $\Ob(C)=\Z_+$,
where the conditions of the theorem are
satisfied. For any  $i\in \Z_+$, we 
shall denote by $[i]$ the set $\{r\in\Z\mid 1\le r\le i\}$;
in particular, $[0]=\emptyset$. 

\begin{example} \label{wreath-product}
Let $\Gamma$ be a finite group.
We define the category $C=\mathrm{FI}_\Gamma$ as follows.
 For any $i, j\in \Z_+$, let
$C(i,j)$ be the set of all pairs $(f, c)$ where
$f: [i]\to [j]$ is an injection, and $c:[i]\to \Gamma$ is an arbitrary map. 
The composition of $(f_1, c_1)\in C(j, l)$ and $(f_2, c_2)\in C(i,j)$ is defined
by 
\begin{equation*}
(f_1, c_1) (f_2, c_2) = (f_3, c_3) 
\end{equation*}
where
\begin{equation*}
f_3(r)=f_1(f_2(r)), \quad c_3(r)=c_1(f_2(r))c_2(r), \quad \mbox{ for all } r\in [i]. 
\end{equation*}
It is easy to see that $C$ is a locally finite EI category 
of type $\Ai$ with isomorphisms
\begin{equation*}
G_i \stackrel{\sim}{\longrightarrow} S_i \ltimes \Gamma^i, \quad
(f,c) \mapsto (f, c(1), \ldots, c(i)) ,
\end{equation*}
where $S_i$ denotes the symmetric group on $[i]$.
We choose $\alpha_i$ to be the pair $(f_i, c_i)$ where $f_i$ is the natural 
inclusion $[i]\hookrightarrow [i+1]$ and $c_i: [i]\to \Gamma$ is the constant
map whose image is the identity element $e$ of $\Gamma$.
Clearly, $C$ satisfies the transitivity condition.
Observe that $H_{i,j}$ is the subgroup of $G_j$ consisting of all pairs $(f,c)$
satisfying $f(r)=r$ and $c(r)=e$ for all $r\in [i]$. 
For any $(f,c) \in C(i,j)$, denote by $f^{-1}[i]$ the set
of $r\in [i]$ such that $f(r)\in [i]$.
Two pairs $(f,c)$, $(g,d)\in C(i,j)$ are in the same $H_{i,j}$-orbit if and only if  
\begin{equation*}
f^{-1}[i]=g^{-1}[i], \quad
f\mid_{f^{-1}[i]}=g\mid_{g^{-1}[i]},  \quad \mbox{ and }\quad
c\mid_{f^{-1}[i]}=d\mid_{g^{-1}[i]} .
\end{equation*}
Let $G'_i$ denote the set of all triples $(U,a,b)$ where $U\subset [i]$,
$a:  U\to [i]$ is an injection, and $b: U\to \Gamma$ is an arbitrary map.
We have an injective map
\begin{equation*}
\theta_{i,j} : H_{i,j}\backslash C(i,j) \longrightarrow G'_i,
\quad H_{i,j} (f,c)\mapsto (f^{-1}[i], f\mid_{f^{-1}[i]}, c\mid_{f^{-1}[i]}) , 
\end{equation*}
When $j\ge 2i$, the map $\theta_{i,j}$ is surjective.
One has $\theta_{i,j+1}  \mu_{i,j} = \theta_{i,j}$. 
Therefore, $\mu_{i,j}$ is bijective when $j\ge 2i$, and hence $C$ 
satisfies the bijectivity condition.
\end{example}

\begin{remark}
When $\Gamma$ is the trivial group, the category FI$_\Gamma$
is equivalent to the category FI; in this case, the observations
in Example \ref{wreath-product} are essentially contained in 
the proof of \cite[Lemma 3.1]{K}. When $\Gamma$ is the cyclic
group of order 2, the category FI$_\Gamma$ is equivalent to the category FI$_{BC}$
in \cite{Wi}. One can similarly show that the category FI$_D$ in \cite{Wi}
satisfies the conditions of Theorem \ref{maintheorem}.
\end{remark}

\begin{example}  \label{GL}
Let $\F_q$ be the finite field with $q$ elements. We define the category $C$
as follows.
For any $i, j \in\Z_+$, let $C(i,j)$ be the set of all injective linear maps 
from $\F_q^i$ to $\F_q^j$. The group $G_i$ is the general linear group $GL_i(\F_q)$.
We choose $\alpha_i$ to be the natural inclusion $\F_q^i\hookrightarrow \F_q^{i+1}$.
Clearly, $C$ is a locally finite EI category of type $\Ai$
satisfying the transitivity  condition.
Let us check the bijectivity condition. 

The map $\alpha_{i,j}$ is the natural inclusion $\F_q^i\hookrightarrow \F_q^j$, and
$H_{i,j}$ is the subgroup of $GL_j(\F_q)$ consisting of all matrices
of form:
\begin{equation*}
\left( \begin{array}{cc} 
I_i & X \\
0 & Y 
\end{array} \right)
\end{equation*}
where $I_i$ denotes the $i$-by-$i$ identity matrix, $X$ is any $i$-by-$(j-i)$ matrix,
and $Y$ is any invertible $(j-i)$-by-$(j-i)$ matrix.
We write the elements of $C(i,j)$ as $j$-by-$i$ matrices. 
For any $A\in C(i,j)$, we denote by $U_A$ the subspace of $\F_q^i$ 
spanned by the last $j-i$ rows of $A$, and 
define $f_A : \F_q^i \to \F_q^i/U_A$ to be the linear map that sends
the $r$-th standard basis vector $e_r$ of $\F_q^i$ to the $r$-th row of $A$
modulo $U_A$, for each $r\in [i]$.
Two elements $A, B \in C(i,j)$ are in the same $H_{i,j}$-orbit if and only if
$U_A=U_B$ and $f_A=f_B$.
Let $G'_i$ be the set of all pairs $(U, f)$ where $U$ is any subspace of $\F_q^i$ 
and $f: \F_q^i \to \F_q^i/U$ is a surjective linear map.
We have an injective map
\begin{equation*}
\theta_{i,j} : H_{i,j}\backslash C(i,j) \longrightarrow G'_i, \quad
H_{i,j} A \mapsto (U_A, f_A). 
\end{equation*}
When $j\ge 2i$, the map $\theta_{i,j}$ is surjective.
One has $\theta_{i,j+1}  \mu_{i,j} = \theta_{i,j}$. 
Therefore, $\mu_{i,j}$ is bijective when $j\ge 2i$.
\end{example}

\begin{remark}
The category $C$ of Example \ref{GL} is equivalent to the
category VI of finite dimensional $\F_q$-vector spaces 
and linear injections. In Example \ref{VIC} below, we consider a variant VIC whose
objects are also the finite dimensional $\F_q$-vector spaces but the morphisms
are pairs $(f,Z)$ where $f$ is an injective linear map and $Z$ is a 
complementary subspace to the image of $f$. The Noetherian property
of VI and VIC over a Noetherian ring $\kk$ are proved by Putman and Sam in \cite{PS};
for VI, it is also proved by Sam and Snowden in \cite{SS2}.
\end{remark}

\begin{example} \label{VIC}
As before, let $\F_q$ be the finite field with $q$ elements. We define the category $C$
as follows.
For any $i, j \in\Z_+$, let $C(i,j)$ be the set of all pairs $(f,Z)$ where $f$ is an
injective linear map from $\F_q^i$ to $\F_q^j$ and $Z\subset \F_q^j$ is a subspace
complementary to the image of $f$. The composition of morphisms is defined by
\begin{equation*}
(f, Z) \circ (f', Z') = (f \circ f', Z + f(Z')).
\end{equation*} 
The group $G_i$ is the general linear group $GL_i(\F_q)$.
Clearly, $C$ is a locally finite EI category of type $\Ai$ satisfying the transitivity condition.
We choose $\alpha_i$ to be the pair $(f_i,Z_i)$ where $f_i$ is the
natural inclusion $\F_q^i\hookrightarrow \F_q^{i+1}$ and $Z_i=\{(0,\ldots,0,z)\mid z\in \F_q\}$,
so $H_{i,j}$ is the subgroup of $GL_j(\F_q)$ consisting of all matrices of the form:
\begin{equation*}
\left( \begin{array}{cc} 
I_i & 0 \\
0 & Y 
\end{array} \right) \quad \mbox{ where } Y\in GL_{j-i}(\F_q).
\end{equation*}
The map 
\begin{equation*}
\mu'_{i,j}: H_{i,j}\backslash GL_j(\F_q)/ H_{i,j} \longrightarrow H_{i,j+1}\backslash GL_{j+1}(\F_q)/ H_{i,j+1}
\end{equation*}
is induced by the standard inclusion of $GL_j(\F_q)$ into $GL_{j+1}(\F_q)$, see Remark \ref{double coset formulation of bijectivity condition}. 
To check the bijectivity condition, it suffices to show that for each $i\in \Z_+$, the maps $\mu'_{i,j}$ are surjective for all $j$ sufficiently large.

Let us show that $\mu'_{i,j}$ is surjective when $j\geqslant 3i$. Let $X\in GL_{j+1}(\F_q)$ where $j\geqslant 3i$.
We need to show that for some $g, g' \in H_{i,j+1}$, the only nonzero entry in the last row or last column of $gXg'$ is the entry in position $(j+1,j+1)$ and it is equal to 1. First, it is easy to see that for suitable choices of $h_1,h_1' \in H_{i,j+1}$, the entries of $h_1 X h_1'$ in positions $(r,s)$ are 0 if $r>2i$ and $s\leqslant i$, or if $r\leqslant i$ and $s>2i$. Indeed, we may first 
perform row operations on the last $j-i$ rows of $X$ to change the entries in  positions $(r,s)$ to 0 for $r>2i$ and $s\leqslant i$, then
perform column operations on the last $j-i$ columns of the resulting matrix to change the entries in positions $(r,s)$ to 0 for $r\leqslant i$ and $s>2i$.
Set $Y=h_1 X h_1'$. Since $Y$ has rank $j+1$, and $j+1>3i$, there exists $r, s>2i$ such that the entry in position $(r,s)$ of $Y$ is nonzero. Swapping row $r$ with row $j+1$, and then swapping column $s$ with column $j+1$, we may assume that the entry of $Y$ in position $(j+1,j+1)$ is nonzero, and by rescaling the entries in the last row we may assume that this entry is 1. It is now easy to see that for some $h_2, h_2'\in H_{i,j}$, the matrix $h_2 Y h_2'$ has the required form.
\end{example}

\section{Proof of main result} \label{proofs}

Recall that $C$ denotes a locally finite EI category of type $\Ai$. 
We assume in this section
that $C$ satisfies the transitivity and bijectivity
conditions, and $\kk$ is a field of characteristic 0.

\subsection{Special case}
Let $i\in \Z_+$. We first prove Theorem \ref{maintheorem} in the
special case when $V$ is $M(i)$. 
By Lemma \ref{bijective-implies-injective}, there exists
$N\in \Z_+$ such that the maps 
$m_{i,j}$ of (\ref{mapm}) are injective and the maps
$\mu_{i,j}$ of (\ref{mapmu}) are bijective for all $j\ge N$. 
We fix such a $N$.

We shall need the following simple observation.

\begin{lemma}
Let $H$ be a finite group and let $\OO_1$ be a set on which $H$ acts transitively.
Let $\OO_2 \subset \OO_1$ be any nonempty subset. Then one has:
\begin{equation*}
 \frac{1}{|H|} \sum_{h\in H} h\left(  \frac{1}{|\OO_2|} \sum_{x \in \OO_2} x \right)
= \frac{1}{|\OO_1|} \sum_{y\in \OO_1} y  
\end{equation*}
in the $\kk H$-module $\kk \OO_1$.
\end{lemma}

\begin{proof}
Let $r$ be the order of $\Stab_H (y)$ for any $y\in \OO_1$. Then one has:
\begin{eqnarray*}
 \frac{1}{|H|} \sum_{h\in H} h\left(  \frac{1}{|\OO_2|} \sum_{x \in \OO_2} x \right)
 &=&  \frac{1}{|\OO_2|} \sum_{x \in \OO_2} \left( \frac{1}{|H|} \sum_{h\in H} hx
 \right)  \\
 &=&  \frac{1}{|\OO_2|} \sum_{x \in \OO_2} \left( \frac{r}{|H|}  \sum_{y\in \OO_1} y \right)  \\
 &=&  \frac{1}{|\OO_1|} \sum_{y\in \OO_1} y .
\end{eqnarray*}
\end{proof}

Now suppose $j\ge N$.
For each $H_{i,j}$-orbit $\OO$ in $C(i,j)$, 
we define a $\kk G_j$-module endomorphism
\begin{equation*}
f_\OO : M(i)_j \to M(i)_j 
\end{equation*}
by
\begin{equation*}
f_\OO(g \alpha_{i,j})   =  \frac{1}{|\OO|} \sum_{\gamma\in \OO}  g \gamma,
\quad \mbox{ for any } g\in G_j.  
\end{equation*}
The elements $f_\OO$ for $\OO\in H_{i,j}\backslash C(i,j)$ form a basis for 
$\End_{\kk G_j}(M(i)_j)$. Hence, 
from the bijectivity of $\mu_{i,j}$, we have a linear bijection
\begin{equation*}
 \nu_{i,j}:
 \End_{\kk G_j}(M(i)_j) \to \End_{\kk G_{j+1}}(M(i)_{j+1}) , \quad  
 f_\OO \mapsto f_{\mu_{i,j}(\OO)}. 
\end{equation*}

\begin{notation}
Define $e_{i,j} \in \kk H_{i,j}$ by
\begin{equation*}
e_{i,j} = \frac{1}{|H_{i,j}|}  \sum_{h\in H_{i,j}} h . 
\end{equation*}
\end{notation}

\begin{lemma}
For any $f\in \End_{\kk G_j} (M(i)_j)$, one has:
\begin{equation} \label{nuf}
\nu_{i,j}(f) (\alpha_{i,j+1}) =
 e_{i,j+1} \alpha_j ( f(\alpha_{i,j})).
\end{equation}
\end{lemma}

\begin{proof}
By linearity, it suffices to verify (\ref{nuf}) when $f=f_\OO$, for each
$\OO \in H_{i,j}\backslash C(i,j)$. 

The injectivity of $m_{i,j}$ implies 
 that $\{\alpha_j \gamma \mid \gamma\in \OO\}$ is a
subset of $\mu_{i,j}(\OO)$ consisting of $|\OO|$ elements.
Thus, by the preceding lemma, one has:
\begin{eqnarray*}
 e_{i,j+1} \alpha_j ( f_\OO(\alpha_{i,j}))  &=&
 \frac{1}{|H_{i,j+1}|} \sum_{h\in H_{i,j+1}} h \left( 
\frac{1}{|\OO|} \sum_{\gamma\in \OO} \alpha_j\gamma
 \right) \\
 &=& \frac{1}{|\mu_{i,j}(\OO)|} \sum_{y\in \mu_{i,j}(\OO)} y \\
 &=&  f_{\mu_{i,j}(\OO)} (\alpha_{i,j+1}) \\
 &=&  \nu_{i,j}(f_\OO) (\alpha_{i,j+1}) .
\end{eqnarray*}
\end{proof}

For any  $\kk G_j$-submodule $U\subset M(i)_j$, we have a natural
inclusion
\begin{equation*}
\Hom_{\kk G_j} (M(i)_j, U) \subset \End_{\kk G_j} (M(i)_j). 
\end{equation*}
By Maschke's theorem,
if $U\subsetneq U'$ are $\kk G_j$-submodules of $M(i)_j$, then
\begin{equation}  \label{Maschke}
\Hom_{\kk G_j}(M(i)_j, U) \subsetneq \Hom_{\kk G_j}(M(i)_j, U') . 
\end{equation}

\begin{lemma}  \label{nufin}
Let $X$ be a $\kk C$-submodule of $M(i)$. 
If $f \in   \Hom_{\kk G_j} (M(i)_j, X_j)$, then 
\begin{equation*}
\nu_{i,j} (f) \in \Hom_{\kk G_{j+1}} (M(i)_{j+1}, X_{j+1}) . 
\end{equation*}
\end{lemma}

\begin{proof}
Since  $f \in   \Hom_{\kk G_j} (M(i)_j, X_j)$, one has 
$f(\alpha_{i,j})\in X_j$. It follows by (\ref{nuf}) that
$\nu_{i,j}(f)(\alpha_{i,j+1}) \in X_{j+1}$.
By the transitivity condition, one has
$\nu_{i,j}(f)(\alpha) \in X_{j+1}$ for all $\alpha\in C(i,j+1)$.
\end{proof}

\begin{definition}
For any $\kk C$-submodule $X$ of $M(i)$, let
\begin{equation*}
F_j (X) = \Hom_{\kk G_j} (M(i)_j, X_j). 
\end{equation*}
\end{definition}

By Lemma \ref{nufin}, we have a commuting diagram:
\begin{equation} \label{commutingdiagram}
 \xymatrix{ F_j (X)  \ar[rr]^{\nu_{i,j}}   \ar@{^{(}->}[d] & & 
F_{j+1}(X) \ar[rr]^{\nu_{i,j+1}}     \ar@{^{(}->}[d]   & &
F_{j+2}(X) \ar[rr]^{\nu_{i,j+2}}    \ar@{^{(}->}[d]   & & \cdots \\
F_j(M(i))  \ar[rr]^{\nu_{i,j}}  & &
F_{j+1}(M(i)) \ar[rr]^{\nu_{i,j+1}} & &
F_{j+2}(M(i)) \ar[rr]^{\nu_{i,j+2}} & & \cdots.
}
\end{equation}
The maps in the bottom row of (\ref{commutingdiagram}) are bijective.
Thus, 
\begin{equation} \label{top-row}
 \mbox{\it the maps in the top row of (\ref{commutingdiagram}) are injective. } 
\end{equation}

\begin{proposition}  \label{submodule-of-M(i)}
If $X$ is a $\kk C$-submodule of $M(i)$,
then $X$ is finitely generated.
\end{proposition}

\begin{proof}

If $n\in\Z_+$, 
we shall write $X(n)$ for the $\kk C$-submodule of $X$ 
generated by  $\bigoplus_{i\le n} X_i$.
It suffices to prove that
$X$ is equal to $X(n)$ for some $n$. 
Suppose not. This means that for each $n$,
there exists $r_n > n$ such that $X(n)_{r_n} \neq X_{r_n}$, and hence
$X(n)_{r_n} \neq X(r_n)_{r_n}$. Consider the sequence
\begin{equation*} 
n_1=N, \quad n_2 = r_{n_1},\quad  n_3 = r_{n_2} , \quad \ldots . 
\end{equation*}
We have 
\begin{equation*}
N= n_1 < n_2 < n_3 < \cdots ,
\end{equation*}
and
\begin{equation*}
X(n_i)_{n_{i+1}} \subsetneq X(n_{i+1})_{n_{i+1}}\quad \mbox{ for all } i. 
\end{equation*}
By (\ref{Maschke}) and (\ref{top-row}), we have 
\begin{multline*}
  F_{n_2}(X(n_1))  
 \subsetneq F_{n_2}(X(n_2)) 
  \hookrightarrow F_{n_3}(X(n_2))  
 \subsetneq F_{n_3}(X(n_3)) \\
  \hookrightarrow F_{n_4}(X(n_3))  
 \subsetneq F_{n_4}(X(n_4))
 \hookrightarrow \cdots 
\end{multline*}
This is an increasing chain which is strictly increasing at every other step. But 
the dimension of each term in this chain is at most $\dim F_N(M(i))$.
Therefore, we have a contradiction.
\end{proof}

\subsection{Proof of Theorem \ref{maintheorem}}

Let $V$ be a finitely generated $\kk C$-module.
Let $Y$ be a $\kk C$-submodule of $V$.
We shall show that $Y$ is finitely generated.

By Lemma \ref{fg}, there exists a finite set $S$ of homogenous elements of $V$ 
such that the homomorphism $\pi_S: M(S) \to V$ of (\ref{pi_S}) is surjective.
Let $X$ be the preimage $\pi_S^{-1}(Y)$ of $Y$. It suffices to prove that 
any $\kk C$-submodule $X$ of $M(S)$ is 
finitely generated. We shall use induction on $|S|$.
If $|S|=1$, the result follows from Proposition \ref{submodule-of-M(i)}.
Suppose $|S|>1$. We choose any $s\in S$ and let $S'=S-\{s\}$; so 
$M(S)=M(S')\oplus M(\deg s)$.
Let $p_s : M(S) \to M(\deg s)$ be the projection map with kernel $M(S')$.
We have a short exact sequence
\begin{equation*}
 0 \longrightarrow X\cap M(S') \longrightarrow  X 
 \stackrel{p_s}{\longrightarrow} p_s(X) \longrightarrow 0. 
\end{equation*}
Since $X\cap M(S')\subset M(S')$ and $p_s(X)\subset M(\deg s)$, it follows 
by induction hypothesis that $X\cap M(S')$ and $p_s(X)$ are both finitely generated.
Hence, $X$ is finitely generated.
This concludes the proof of Theorem \ref{maintheorem}.     

\begin{remark}
After this paper was written, Andrew Putman informed us that he had also found a 
similar proof for the categories FI and VI.
\end{remark}

\section{Further remarks} \label{further-remarks}

In this section, we discuss generalizations of
the injectivity and surjectivity properties
of finitely generated FI-modules; see \cite[Definition 3.3.2]{CEF}.
Let $C$ be a locally finite EI category of type $\Ai$,
and $\kk$ be any commutative ring.

\begin{proposition}  \label{injective}
Assume that $C$ satisfies the transitivity condition. 
Let $V$ be a Noetherian $\kk C$-module. Then the map
$\alpha_j : V_j \to V_{j+1}$ is injective for all $j$ sufficiently large.
\end{proposition}

\begin{proof}
Let $X_j$ be the kernel of the map $\alpha_j : V_j \to V_{j+1}$.
Suppose $g\in G_j$.
By the transitivity condition, there exists 
$g_1 \in G_{j+1}$ such that 
$g_1 \alpha_j = \alpha_j g$.
If $x\in X_j$, then one has $\alpha_j(gx) = g_1\alpha_j(x) = 0$.
Hence, $X_j$ is a $\kk G_j$-submodule of $V_j$.
For any $\beta\in C(j,j+1)$, there exists $g_2\in G_j$ such that
$g_2\alpha_j = \beta$; hence, for any $x\in X_j$, one has
$\beta(x) = g_2\alpha_j(x) = 0$. It follows that $X_j$ is 
a $\kk C$-submodule of $V$. Let
\begin{equation*}
X = \bigoplus_{j \in \Z_+} X_j. 
\end{equation*}
Then $X$ is a $\kk C$-submodule of $V$  (called the \emph{torsion submodule} of $V$, see \cite{CEFN}).
By hypothesis, $X$ is finitely generated. Since each $X_j$ is a 
$\kk C$-submodule of $X$, it follows that $X_j$ must be zero for
all $j$ sufficiently large.
\end{proof}

For any $\kk C$-module $V$, we shall denote by $\rho_j(V)$ the image
of the $\kk G_{j+1}$-module map
\begin{equation*}
\kk C(j,j+1) \otimes_{\kk G_j} V_j 
\longrightarrow V_{j+1},  \quad
\alpha\otimes v \mapsto \alpha v . 
\end{equation*}

\begin{proposition}
Assume that $C$ satisfies the transitivity condition. 
Let $V$ be a $\kk C$-module. Suppose $V_j$ is 
a finitely generated $\kk$-module for all $j\in\Z_+$.
Then $V$ is finitely generated as a $\kk C$-module
if and only if $\rho_j(V)=V_{j+1}$
for all $j$ sufficiently large.
\end{proposition}

\begin{proof}
Suppose that $V$ is finitely generated.
By Lemma \ref{fg}, there exists a finite set $S$ of homogeneous elements of $V$
such that the
homomorphism $\pi_S:M(S)\to V$ of (\ref{pi_S}) is surjective.  
Let $n$ be the maximal of $\deg s$ for $s\in S$. 
Suppose $j\ge n$. Then one has 
\begin{equation*}
\rho_j(V) = \pi_S(\rho_j(M(S))) = \pi_S(M(S)_{j+1})= V_{j+1}. 
\end{equation*}

Conversely, suppose there exists $n\in\Z_+$ such
that $\rho_j(V)=V_{j+1}$ for all $j\ge n$.
For each $i\le n$, let $B_i$ be a finite subset of $V_i$ which 
generates $V_i$ as a $\kk$-module.
Let $S=B_0 \cup \cdots \cup B_n$. Then $S$ is a finite set of homogeneous elements
of $V$. 
Let $V'$ be the $\kk C$-submodule of $V$ generated by $S$. 
Clearly, $V'_i = V_i$ for all $i\le n$. 
Suppose, for induction, that $V'_j =V_j$ for some $j\ge n$. Then one has
\begin{equation*}
V'_{j+1} \supset  \rho_j (V') = \rho_j (V) = V_{j+1}.
\end{equation*}
Hence, $V'_{j+1} = V_{j+1}$. It follows that $V'=V$, so $V$ is finitely generated.
\end{proof}


\begin{thebibliography}{99}

\bibitem{CEF} T. Church, J. Ellenberg, B. Farb,
FI-modules and stability for representations of symmetric groups,
to appear in Duke Math J., arXiv:1204.4533.
 
 
\bibitem{CEFN}  T. Church, J. Ellenberg, B. Farb, R. Nagpal, 
FI-modules over Noetherian rings, 
Geom. Top. 18-5 (2014), 2951-2984, arXiv:1210.1854.

\bibitem{F} B. Farb, 
Representation stability, 
to appear in Proceedings of ICM 2014, arXiv:1404.4065.

\bibitem{K} A. Krieg,
Hecke algebras, Mem. Amer. Math. Soc.  87 (1990), no. 435. 

\bibitem{L} W. L\"uck,
Transformation groups and algebraic K-theory,
Lecture Notes in Mathematics vol. 1408, Springer-Verlag, 1989.

\bibitem{PS} A. Putman, S. Sam,
Representation stability and finite linear groups,
arXiv:1408.3694.

\bibitem{SS1} S. Sam, A. Snowden,
GL-equivariant modules over polynomial rings in infinitely many variables,
to appear in Trans. Amer. Math. Soc., arXiv:1206.2233.

\bibitem{SS2} S. Sam, A. Snowden, 
Gr\"{o}bner methods for representations of combinatorial categories, arXiv:1409.1670.

\bibitem{Sn} A. Snowden,
Syzygies of Segre embeddings and $\Delta$-modules,
Duke Math. J. 162 (2013) 2, 225--277, arXiv:1006.5248.

\bibitem{Wi} J. Wilson, 
FI$_{\mathcal{W}}$-modules and stability criteria for representations of classical Weyl groups, J. Algebra 420 (2014), 269-332, arXiv:1309.3817.



\end{thebibliography}
\end{document}